\documentclass[12pt]{amsart}
\usepackage[all]{xy}
\usepackage[parfill]{parskip}

\usepackage{verbatim}
\usepackage{color}

\usepackage{amsmath, amscd, graphicx, latexsym, hyperref, times}
\usepackage{color,graphicx}
\usepackage[abs]{overpic}
\usepackage{tikz}
\usepackage{tikz-cd}
\usetikzlibrary{arrows, patterns}

\oddsidemargin .25in \theoremstyle{plain}
\newtheorem{theorem}{Theorem}

\newtheorem{lemma}{Lemma}
\newtheorem{proposition}{Proposition}

\newtheorem{conjecture}{Conjecture}

\newtheorem{remark}{Remark}

\numberwithin{equation}{section}
\numberwithin{lemma}{section}
\numberwithin{proposition}{section}
\numberwithin{corollary}{section}
\numberwithin{remark}{section}
\allowdisplaybreaks[3]

\usepackage[backend=biber,style=alphabetic]{biblatex} 
\usepackage{cleveref}
\crefname{Theorem}{Theorem}{Theorem}
\crefname{lemma}{Lemma}{Lemma}
\crefname{proposition}{Proposition}{Proposition}
\crefname{conjecture}{conjecture}{conjecture}
\crefname{equation}{}{}
\crefname{section}{Section}{Section}
\crefname{figure}{Figure}{Figure}
\usepackage{appendix}
\begin{document}

\title{A note on Wang's conjecture for harmonic functions with nonlinear boundary condition}
\author{Xiaohan Cai}
\address{School of Mathematical Sciences, Shanghai Jiao Tong University}
\email{xiaohancai@sjtu.edu.cn}

\thanks{}
\date{}

\begin{abstract}
We obtain some Liouville type theorems for positive harmonic functions on compact Riemannian manifolds with nonnegative Ricci curvature and strictly convex boundary and partially verifies Wang's conjecture (J. Geom. Anal. \textbf{31} (2021)). 

For the specific manifold $\mathbb{B}^n$, we present a new proof of this conjecture, which has been resolved by Gu-Li (Math. Ann. \textbf{391} (2025)). Our proof is based on a general principle of applying the P-function method to such Liouville type results. As a further application of this method, we obtain some classification results for nonnegative solutions of some semilinear elliptic equations  with  a nonlinear boundary condition.
\end{abstract}
\maketitle
\section{Introduction}

Wang proposed a conjecture for Liouville type result for harmonic functions with some specific nonlinear boundary condition \cite[Conjecture 1]{Wan21}.  
\begin{conjecture}[Wang, \cite{Wan21}]\label{conj. Wang}
    Let $(M^n,g)\ (n\geq 3)$ be a compact Riemannian manifold with $\mathrm{Ric}\geq 0$ and the second fundamental form $\Pi\geq 1$ on $\partial M$. If $u\in C^{\infty}(M)$ is a positive solution of the following equation
    \begin{align}\label{eq. equation of u on M}
        \begin{cases}
            \Delta u=0  &\text{in }M,\\
            \frac{\partial u}{\partial \nu}+\lambda u=u^q &\text{on }\partial M,
        \end{cases}
    \end{align}
    where $1<q\leq \frac{n}{n-2}$ and $0<\lambda\leq \frac{1}{q-1}$ are constants. Then either $u$ is a constant function, or $q=\frac{n}{n-2},\, \lambda=\frac{n-2}{2}$, $(M^n,g)$ is isometric to $\mathbb{B}^n\subset\mathbb{R}^n$ and $u$ corresponds to
    \begin{align*}
            u(x)=\left(
            \frac{n-2}{2}\frac{1-|a|^2}{|a|^2|x|^2-2\langle x,a\rangle+1}
            \right)^{\frac{n-2}{2}},
        \end{align*}
        where $a\in \mathbb{B}^n$.
\end{conjecture}
This conjecture, if proved to be true, have several interesting geometric consequences, such as a sharp upper bound of the area of the boundary and a sharp lower bound of Steklov eigenvalue on such manifolds. See \cite[Section 2]{Wan21} and also \cite[Section 5]{GHW21} for detailed discussions.

Inspired by Xia-Xiong's work on Steklov eigenvalue estimate \cite{XX24}, Guo-Hang-Wang \cite[Theorem 2]{GHW21} verified \cref{conj. Wang} for some special cases under nonnegative sectional curvature condition:
\begin{theorem}[Guo-Hang-Wang, \cite{GHW21}]\label{thm. GHW21}
    Let $(M^n,g)\ (n\geq 3)$ be a compact Riemannian manifold with $\mathrm{Sec}\geq 0$ and $\Pi\geq 1$ on $\partial M$. Then the only positive solution to \cref{eq. equation of u on M} is constant provided $3\leq n\leq 8$ and 
    $1<q\leq \frac{4n}{5n-9}$, $0<\lambda\leq \frac{1}{q-1}$.
\end{theorem}

In the first part of this note, we shall partially confirm \cref{conj. Wang} under Ricci curvature condition via integration by parts and meticulously choosing the parameters.
\begin{theorem}\label{thm. C. Ricci}
    Let $(M^n,g)\ (n\geq 3)$ be a compact Riemannian manifold with $\mathrm{Ric}\geq 0$ and $\Pi\geq 1$ on $\partial M$. Then the only positive solution to \cref{eq. equation of u on M} is constant provided one of the following two conditions holds:
    \begin{enumerate}
        \item $3\leq n\leq 7$, $1<q\leq \frac{3n}{4(n-2)}$, and $0<\lambda\leq \min\{\frac{1}{2(q-1)}, \frac{3(n-1)}{2q}\}$\\
        \item $3\leq n\leq 9$, $1<q\leq \frac{n+1+\sqrt{(5n-1)(n-1)}}{4(n-2)}$
        and $0<\lambda\leq \min\{\frac{1}{2(q-1)},\frac{6q+1}{2q+1}\frac{n-1}{2q}\}$
    \end{enumerate}
\end{theorem}
\begin{remark}
    Let us compare the ranges of $q$ and $\lambda$ in \cref{thm. C. Ricci} with those in \cref{thm. GHW21}.
    \begin{itemize}
    \item Condition (2) in \cref{thm. C. Ricci} allows one more dimension $n=9$ compared with \cref{thm. GHW21}.
        \item \cref{thm. C. Ricci} yields a larger range of $q$ than \cref{thm. GHW21} in some dimensions. Explicitly, we have   $\frac{n+1+\sqrt{(5n-1)(n-1)}}{4(n-2)}\geq \frac{3n}{4(n-1)}\geq \frac{4n}{5n-9}$ if  $3\leq n\leq 5$, and $\frac{n+1+\sqrt{(5n-1)(n-1)}}{4(n-2)}
        \geq \frac{4n}{5n-9}
        \geq \frac{3n}{4(n-1)}$ if $6\leq n\leq 9$.
        \item The range of $\lambda$ in \cref{thm. C. Ricci} is not sharp. Moreover, condition (2) doesn't cover condition (1) since $\frac{6q+1}{2q+1}\frac{n-1}{2 q}<\frac{3(n-1)}{2q}$.
    \end{itemize}
\end{remark}
\vspace{1em}

	Except for studying it within some special ranges of parameters,  another way toward \cref{conj. Wang} is to confine ourselves to some specific manifolds. Guo and Wang  proposed such an individual conjecture on the model space $\mathbb{B}^n$  \cite[Conjecture 1]{GW20}:
    \begin{conjecture}[Guo-Wang, \cite{GW20}]\label{conj.Guo-Wang}
        If $u\in C^{\infty}(\mathbb{B}^n)$ is a positive solution of the following equation
        \begin{equation}\label{eq. harmonic}
            \begin{cases}
                \Delta u=0 &\mathrm{ in} \ \mathbb{B}^n,\\
                \frac{\partial u}{\partial\nu}+\lambda u=u^q &\mathrm{on}\ \mathbb{S}^{n-1},
            \end{cases}
        \end{equation}
        where $1<q\leq\frac{n}{n-2}$ and $0<\lambda \leq \frac{1}{q-1}$ are constants. Then either $u$ is a constant function, or $q=\frac{n}{n-2},\, \lambda=\frac{n-2}{2}$ and $u$ corresponds to
    \begin{align}\label{eq. model solution}
            u(x)=\left(
            \frac{n-2}{2}\frac{1-|a|^2}{|a|^2|x|^2-2\langle x,a\rangle+1}
            \right)^{\frac{n-2}{2}},
        \end{align}
        where $a\in \mathbb{B}^n$.
    \end{conjecture}

    Historically, Escobar \cite[Theorem 2.1]{Esc90} (see also \cite{Esc88}) classified all solution of \cref{eq. harmonic} by Obata's method \cite{Oba71} when $q=\frac{n}{n-2}$ and $\lambda=\frac{n-2}{2}$. After several works in this field \cite{GW20,GHW21,LO23,Ou24}, Gu-Li \cite[Theorem 1.1]{GL25} finally give an affirmative answer to \cref{conj.Guo-Wang}.

    \begin{theorem}[Gu-Li, \cite{GL25}]\label{thm. GL25}
        If $u\in C^{\infty}(\mathbb{B}^n)$ is a positive solution of the  equation
        \cref{eq. harmonic}
        for some constants $1<q<\frac{n}{n-2}$ and $0<\lambda \leq \frac{1}{q-1}$, then $u$ is a constant function.
    \end{theorem}

    Gu-Li's method is based on sophisticated integration by parts, with several delicately chosen parameters, and the computation therein is more or less formidable.

    In the second part of this note, we shall provide a simplified proof of \cref{thm. GL25}. We have to admit that our proof is essentially equivalent to Gu-Li's original proof. However, our argument is based on a \emph{general principle} for classifying solutions of such semilinear elliptic equations, which is a continuation and development of that in \cite{Wan22}. 
    
    The basic idea is to start with the critical power case (i.e. $q=\frac{n}{n-2}$) and study the model solution \cref{eq. model solution} and come up with an appropriate function, known as the P-function in literature (in honer of L. Payne), whose constancy implies the rigidity of the solution $u$. For the subcritical power case (i.e. $1<q<\frac{n}{n-2}$), we regard the equation \cref{eq. harmonic} as the critical one in a larger dimension space (see \cref{eq. intrinsic dimension}), then a modified argument as in the critical power case implies the conclusion. Readers interested in the P-function method are invited to \cite{Pay68,Wei71,Dan11,CFP24} for more research in  this realm.

Specific to \cref{conj.Guo-Wang},  one merit of our argument is that the choice of the parameters is naturally indicated from the viewpoint of the P-function method. As an advantage, the computation is streamlined. Another contribution is that our calculation is in the spirit of Escobar's work \cite{Esc90} and clarifies the role of the weight function as providing a closed conformal vector field (see \cref{cor. div 1w} and \cref{eq. 1w_1-d}). This could shed light on some key difficulties for resolving \cref{conj. Wang}.

    Finally, we mention that Escobar \cite[Theorem 2.1]{Esc90} also classified all conformal metrics on $\mathbb{B}^n$ with nonzero constant scalar curvature and constant boundary mean curvature by studying the solution of the following semilinear elliptic equation with a nonlinear boundary condition:
    \begin{equation*}
            \begin{cases}
                -\Delta u=\frac{n-2}{4(n-1)}R\,u^{\frac{n+2}{n-2}} &\mathrm{ in} \ \mathbb{B}^n,\\
                \frac{\partial u}{\partial\nu}+\frac{n-2}{2} u=\frac{n-2}{2(n-1)}Hu^{\frac{n}{n-2}} &\mathrm{on}\ \mathbb{S}^{n-1},
            \end{cases}
        \end{equation*}
    where $R,\, H$ are constants. The same strategy as our proof of \cref{thm. GL25} could be applied to derive  a slightly more general classification result for such kind of semilinear elliptic equations.
    \begin{theorem}\label{thm. C.}
        If $u\in C^{\infty}(\mathbb{B}^n)$ is a \textbf{nonnegative} solution of the following equation
        \begin{equation}\label{eq. general equation of u-p}
            \begin{cases}
                -\Delta u=\frac{n-2}{4(n-1)}R\,u^{p} &\mathrm{ in} \ \mathbb{B}^n,\\
                \frac{\partial u}{\partial\nu}+\lambda u=\frac{n-2}{2(n-1)}Hu^{\frac{p+1}{2}} &\mathrm{on}\ \mathbb{S}^{n-1},
            \end{cases}
        \end{equation}
        where $R>0, H\geq 0$ are constants, $1<p\leq \frac{n+2}{n-2}$ and $0\leq \lambda\leq\frac{2}{p-1}$.
        \begin{enumerate}
            \item If  $1<p\leq \frac{n+2}{n-2}$ and $\lambda=0$, then $u\equiv 0$.
            \item If  $p=\frac{n+2}{n-2}$ and  $0<\lambda< \frac{2}{p-1}=\frac{n-2}{2}$, then 
            \begin{align*}
                u(x)=\left(\frac{n(n-1)}{R}\right)^{\frac{n-2}{4}}
                \left(
                \frac{1}{2\epsilon}|x|^2+\frac{\epsilon}{2}
                \right)^{-\frac{n-2}{2}},
            \end{align*}
            where $0<\epsilon=\sqrt{\frac{n(n-1)}{R}}
            \left(
            \frac{n-2}{2\lambda}\frac{H}{n-1}+\sqrt{(\frac{n-2}{2\lambda})^2(\frac{H}{n-1})^2+\frac{R}{n(n-1)}(\frac{n-2}{\lambda}-1)}
            \right)$.
            \item If  $p=\frac{n+2}{n-2}$ and $\lambda=\frac{2}{p-1}=\frac{n-2}{2}$, then 
            \begin{align*}
                u(x)
                =\left(\frac{n(n-1)}{R}\right)^{\frac{n-2}{4}}
                \left(
\frac{(1+\epsilon^2|a|^2)|x|^2+2(1+\epsilon^2)\langle x,a\rangle+(\epsilon^2+|a|^2)}{2\epsilon(1-|a|^2)}
            \right)^{-\frac{n-2}{2}},
            \end{align*}
            where $a\in \mathbb{B}^n$, $0<\epsilon
            =\sqrt{\frac{n(n-1)}{R}}
            \left(\frac{H}{n-1}+\sqrt{(\frac{H}{n-1})^2+\frac{R}{n(n-1)}}\right)$.
        \end{enumerate}
    \end{theorem}
    \begin{remark}
        The case $R=0$ is completely solved by \cref{thm. GL25}, so we don't include it in the above theorem.
    \end{remark}
    \begin{remark}
        The case $R>0,\, H=0$ has been obtained by Dou-Hu-Xu \cite[Theorem 1.1]{DHX}. Our arguments present a simplified proof of their result. We also invite readers to \cite{DHX} for some consequences of such a Liouville type result.
    \end{remark}

    We also mention that our strategy for \cref{thm. GL25} could be applied to give another proof of \cite[Theorem 6.1]{BVV91} and hence answer a question raised by Wang \cite[Section 5]{Wan22}. This is not the main theme of this note, so we put it in the appendix. We hope that this strategy and general principle may also be useful in some other situations.

    This note is organized as follows. In \cref{sec2}, we prove \cref{thm. C. Ricci}. In \cref{sec3}, we present our new proof of \cref{thm. GL25}. In \cref{sec4}, we prove \cref{thm. C.}. Finally, in \cref{appendix}, we include our proof of \cite[Theorem 6.1]{BVV91}.
 
\textbf{Acknowledgement:} This work was completed during a research visit to Michigan State University, supported by the Zhiyuan Scholarship of Shanghai Jiao Tong University. The author wishes to thank the university for its hospitality and is especially grateful to Prof. Xiaodong Wang and Dr. Zhixin Wang for their helpful discussions.
\section{Proof of \cref{thm. C. Ricci}}\label{sec2}
Our proof exploits two main ingredients in the method of integration by parts: Bochner formula (Step 1) and the equation itself (Step 2).
\begin{proof}\textbf{Step 1:}
    Let $v:=u^{-\frac{1}{a}}$, where $a>0$ is to be determined. We follow the notations in \cite{GHW21} and define $\chi:=\frac{\partial v}{\partial \nu},\, f:=v|_{\partial M}$. Then it's straightforward to see that
    \begin{align}\label{eq. equation of v on M}
        \begin{cases}
            \Delta v=(a+1)v^{-1}|\nabla v|^2 &\text{in }M,\\
            \chi=\frac{1}{a}(\lambda f-f^{a+1-aq}) &\text{on }\partial M.
        \end{cases}
    \end{align}
    By Bochner formula, there holds
    \begin{align*}
        \frac{1}{2}\Delta |\nabla v|^2
        =\left|\nabla^2 v-\frac{\Delta}{n}g\right|^2
        +\frac{1}{n}(\Delta v)^2
        +\langle \nabla \Delta v,\nabla v\rangle
        +\mathrm{Ric}(\nabla v,\nabla v).
    \end{align*}
    Multiply both sides by $v^b$ and integrate it over $M$, where $b$ is a constant to be determined, we have
    \begin{align}\label{eq. Bochner 1}
        \frac{1}{2}\int_M v^b\Delta |\nabla v|^2
        =\int_Mv^b\left(\left|\nabla^2 v-\frac{\Delta}{n}g\right|^2+\mathrm{Ric}(\nabla v,\nabla v) \right)
        +\frac{1}{n}\int_M v^b(\Delta v)^2
        +\int_M v^b\langle \nabla \Delta v,\nabla v\rangle
        .
    \end{align}
    It follows from \cref{eq. equation of v on M}, divergence theorem and the boundary curvature assumption $\Pi\geq 1, H\geq n-1$ that
    \begin{align}\label{eq. LHS of 1}
        &\frac{1}{2}\int_M v^b\Delta |\nabla v|^2
        =\frac{1}{2}\int_M \mathrm{div}(v^b\nabla|\nabla v|^2)
        -\langle\nabla v^b,\nabla |\nabla v|^2\rangle\notag\\
        =&\int_{\partial M}f^b\langle\nabla_{\nabla f+\chi\nu}\nabla v,\nu\rangle
        -b\int_Mv^{b-1}\langle\nabla_{\nabla v}\nabla v,\nabla v\rangle\notag\\
        =&\int_{\partial M}f^b\left(
        \langle\nabla f,\nabla \chi\rangle-\langle\nabla v,\nabla_{\nabla f}\nu\rangle+\chi\nabla^2v(\nu,\nu)
        \right)
        -b\int_Mv^{b-1}\langle\nabla_{\nabla v}\nabla v,\nabla v\rangle\notag\\
        =&\int_{\partial M}f^b\left(
        \langle\nabla f,\nabla \chi\rangle
        -\Pi(\nabla f,\nabla f)
        +\chi(\Delta v-\Delta f-H\chi)
        \right)
        -b\int_Mv^{b-1}\langle\nabla_{\nabla v}\nabla v,\nabla v\rangle\notag\\
        =&\int_{\partial M}f^b
        \left(
        2\langle\nabla f,\nabla \chi\rangle
        -\Pi(\nabla f,\nabla f)
        +(a+1)\chi f^{-1}(|\nabla f|^2+\chi^2)
        +bf^{-1}\chi|\nabla f|^2
        -H\chi^2
        \right)\notag\\
        &-b\int_Mv^{b-1}\langle\nabla_{\nabla v}\nabla v,\nabla v\rangle\notag\\
        \leq&(2(q-1)\lambda-1)\int_{\partial M}f^b|\nabla f|^2
        +\left(a+1+b+2(a+1-aq)\right)\int_{\partial M}f^b\frac{\chi}{f}|\nabla f|^2
        +(a+1)\int_{\partial M}f^b\frac{\chi}{f}\chi^2\notag\\
        &-(n-1)\int_{\partial M}f^b\chi^2
        -b\int_M v^{b-1}\langle\nabla_{\nabla v}\nabla v,\nabla v\rangle\notag\\
        =&(2(q-1)\lambda-1)\int_{\partial M}f^b|\nabla f|^2
        +\frac{a+1+b+2(a+1-aq)}{a}\lambda\int_{\partial M}f^b|\nabla f|^2\notag\\
        &-\frac{a+1+b+2(a+1-aq)}{a}\int_{\partial M}f^{b+a-aq}|\nabla f|^2
        +\frac{a+1}{a}\lambda\int_{\partial M}f^b\chi^2
        -\frac{a+1}{a}\int_{\partial M}f^{b+a-aq}\chi^2\notag\\
        &-(n-1)\int_{\partial M}f^b\chi^2
        -b\int_M v^{b-1}\langle\nabla_{\nabla v}\nabla v,\nabla v\rangle\notag\\
        =&\left(\frac{a+b+3}{a}\lambda-1\right)
        \int_{\partial M}f^b|\nabla f|^2
        +\left(\frac{a+1}{a}\lambda-(n-1)\right)
        \int_{\partial M}f^b\chi^2\notag\\
        &+\left(2(q-1)-\frac{a+b+3}{a}\right)
        \int_{\partial M}f^{b+a-aq}|\nabla f|^2
        -\frac{a+1}{a}\int_{\partial M}f^{b+a-aq}\chi^2
        -b\int_M v^{b-1}\langle\nabla_{\nabla v}\nabla v,\nabla v\rangle.
    \end{align}
    On the other hand, the right hand side of \cref{eq. Bochner 1} could be written as
    \begin{align}\label{eq. RHS of 1}
         &\int_M v^b\left(
        \left|\nabla^2 v-\frac{\Delta v}{n}g\right|^2
        +\mathrm{Ric}(\nabla v,\nabla v)
        \right)
        +\frac{(a+1)^2}{n}\int_M w v^{b-2}|\nabla v|^4
        -(a+1)\int_M v^{b-2}|\nabla v|^4\notag\\
        &+2(a+1)\int_M v^{b-1}\langle\nabla_{\nabla v}\nabla v,\nabla v\rangle.
    \end{align}
    Therefore, by \cref{eq. Bochner 1}, \cref{eq. LHS of 1} and \cref{eq. RHS of 1} we have
    \begin{align}\label{eq. equation 1}
        &\left(\frac{a+b+3}{a}\lambda-1\right)
        \int_{\partial M}f^b|\nabla f|^2
        +\left(\frac{a+1}{a}\lambda-(n-1)\right)
        \int_{\partial M}f^b\chi^2\notag\\
        &+\left(2(q-1)-\frac{a+b+3}{a}\right)\int_{\partial M}f^{b+a-aq}|\nabla f|^2\notag
        -\frac{a+1}{a}\int_{\partial M}f^{b+a-aq}\chi^2\\
        \geq & \int_M v^b\left(
        \left|\nabla^2 v-\frac{\Delta v}{n}g\right|^2
        +\mathrm{Ric}(\nabla v,\nabla v)
        \right)
        +\left(\frac{(a+1)^2}{n}-(a+1)\right)\int_M v^{b-2}|\nabla v|^4\notag\\
        &+(2(a+1)+b)\int_M v^{b-1}\langle\nabla_{\nabla v}\nabla v,\nabla v\rangle.
    \end{align}

    \textbf{Step 2:} Multiply both sides of \cref{eq. equation of v on M} by $v^b$ and integrate it over $M$, we  have
    \begin{align}\label{eq. "Bochner" 2}
        \int_{M}v^b(\Delta v)^2=(a+1)\int_M v^{b-1}|\nabla v|^2\Delta v.
    \end{align}
    It follows from \cref{eq. equation of v on M} that the left hand side of \cref{eq. "Bochner" 2} could be written as
    \begin{align}\label{eq. LHS of 2}
        &\int_{M}v^b(\Delta v)^2
        =\int_M \mathrm{div}(v^b\Delta v\nabla v)
        -\langle \nabla v^b,\nabla v\rangle\Delta v
        -\langle\nabla \Delta v,\nabla v\rangle v^b\notag\\
        =&(a+1)\int_{\partial M}f^{b-1}(|\nabla f|^2+\chi^2)\chi
        -(a+1)b\int_Mv^{b-2}|\nabla v|^4\notag\\
       & +(a+1)\int_M v^{b-2}|\nabla v|^4
        -2(a+1)\int_M v^{b-1}\langle\nabla_{\nabla v}\nabla v,\nabla v\rangle\notag\\
        =&\frac{a+1}{a}\lambda\int_{\partial M}f^b|\nabla f|^2
        +\frac{a+1}{a}\lambda\int_{\partial M}f^b\chi^2
        -\frac{a+1}{a}\int_{\partial M}f^{b+a-aq}|\nabla f|^2
        -\frac{a+1}{a}\int_{\partial M}f^{b+a-aq}\chi^2\notag\\
        &+(a+1)(1-b)\int_M v^{b-2}|\nabla v|^4
        -2(a+1)\int_M v^{b-1}\langle\nabla_{\nabla v}\nabla v,\nabla v\rangle.
    \end{align}
    Therefore, by \cref{eq. "Bochner" 2} and \cref{eq. LHS of 2} we have
    \begin{align}\label{eq. equation 2}
        &\frac{a+1}{a}\lambda\int_{\partial M}f^b|\nabla f|^2
        +\frac{a+1}{a}\lambda\int_{\partial M}f^b\chi^2
        -\frac{a+1}{a}\int_{\partial M}f^{b+a-aq}|\nabla f|^2
        -\frac{a+1}{a}\int_{\partial M}f^{b+a-aq}\chi^2\notag\\
        =&\left(
        (a+1)^2+(a+1)(b-1)
        \right)\int_M v^{b-2}|\nabla v|^4
        +2(a+1)\int_M v^{b-1}\langle\nabla_{\nabla v}\nabla v,\nabla v\rangle.
    \end{align}
    \textbf{Step 3:} Now consider $\cref{eq. equation 1}+c\cref{eq. equation 2}$:
    \begin{align*}
        &\left(\frac{a+b+3}{a}\lambda-1+\frac{a+1}{a}\lambda c\right)
        \int_{\partial M}f^b|\nabla f|^2
        +\left(\frac{a+1}{a}\lambda-(n-1)+\frac{a+1}{a}\lambda c\right)
        \int_{\partial M}f^b\chi^2\notag\\
        &+\left(2(q-1)-\frac{a+b+3}{a}-\frac{a+1}{a}c\right)
        \int_{\partial M}f^{b+a-aq}|\nabla f|^2\notag
        -(\frac{a+1}{a}+\frac{a+1}{a}c)\int_{\partial M}f^{b+a-aq}\chi^2\\
        \geq & \int_M v^b\left(
        \left|\nabla^2 v-\frac{\Delta v}{n}g\right|^2
        +\mathrm{Ric}(\nabla v,\nabla v)
        \right)
        +\left((\frac{1}{n}+c)(a+1)^2
        +(cb-c-1)(a+1)\right)\int_M v^{b-2}|\nabla v|^4\notag\\
        &+\left(2(c+1)(a+1)+b\right)\int_M v^{b-1}\langle\nabla_{\nabla v}\nabla v,\nabla v\rangle.
    \end{align*}
   Define $\beta$ by setting $b=-\beta(a+1)$, $x:=\frac{1}{a}$ and choose $c=-1+\frac{\beta}{2}$ to eliminate the  term $\int_M v^{b-1}\langle \nabla_{\nabla v}\nabla v,\nabla v\rangle$, we obtain
    \begin{align}\label{eq. Final ineq.}
        &\left(\frac{4-\beta}{2}\lambda x-(1+\frac{\beta\lambda}{2})\right)
        \int_{\partial M}f^b|\nabla f|^2
        +\left(\frac{\beta\lambda}{2}(x+1)-(n-1)\right)
        \int_{\partial M}f^b\chi^2\notag\\
        &+\left(-\frac{4-\beta}{2}x+2(q-1)+\frac{1}{2}\beta\right)
        \int_{\partial M}f^{b+a-aq}|\nabla f|^2\notag
        -\frac{1}{2}\beta(x+1)\int_{\partial M}f^{b+a-aq}\chi^2\\
        \geq & \int_M v^b\left(
        \left|\nabla^2 v-\frac{\Delta v}{n}g\right|^2
        +\mathrm{Ric}(\nabla v,\nabla v)
        \right)\notag\\
        &+\frac{1}{x^2}(x+1)\left(
        -(\frac{1}{2}\beta^2-\beta+\frac{n-1}{n})(x+1)+\frac{1}{2}\beta
        \right)
        \int_M v^{b-2}|\nabla v|^4.
    \end{align}
    \textbf{Step 4:} Now we verify the condition (1) in \cref{thm. C. Ricci}: Take $\beta=1, x=\frac{\beta+4(q-1)}{4-\beta}=\frac{1+4(q-1)}{3}$ such that $-\frac{4-\beta}{2}x+2(q-1)+\frac{1}{2}\beta=0$. Equivalently, $a=\frac{3}{1+4(q-1)}$ and $b=-(a+1)=-\frac{4q}{4q-3}$. Then \cref{eq. Final ineq.} turns out to be
    \begin{align}\label{eq. Final (a)}
        &\left(2(q-1)\lambda-1 \right)\int_{\partial M}f^b|\nabla f|^2
        +\left(\frac{2}{3}q\lambda-(n-1)\right)\int_{\partial M}f^b\chi^2
        -\frac{2}{3}q\int_{\partial M}f^{b+a-aq}\chi^2\notag\\
        \geq& \int_M v^b\left(
        \left|\nabla^2 v-\frac{\Delta v}{n}g\right|^2
        +\mathrm{Ric}(\nabla v,\nabla v)
        \right)
        +\left(\frac{3}{1+4(q-1)}\right)^2
        \frac{4}{3}q
        \left(\frac{1}{2}-\frac{2(n-2)}{3n}q\right)
        \int_M v^{b-2}|\nabla v|^4.
    \end{align}
    Then the condition (1) in \cref{thm. C. Ricci} implies that the left hand side of \cref{eq. Final (a)} is less or equal to zero, while the right hand side of \cref{eq. Final (a)} is larger or equal to zero. It follows that $0\equiv\chi=\frac{\partial v}{\partial \nu}$. Hence by \cref{eq. equation of v on M} we deduce that $v|_{\partial M}=f$ is constant, and $u$ is a harmonic function in $M$ with constant boundary value on $\partial M$. Therefore $u$ is constant in $M$.

    \textbf{Step 5:} Now we verify the condition (2) in \cref{thm. C. Ricci}: By analyzing the range of $x$ such that the left hand side of \cref{eq. Final ineq.} is less or equal to zero and the right hand side of \cref{eq. Final ineq.} is larger or equal to zero, we obtain the optimal value of $\beta$ as $\beta=\frac{4q+2}{4q+1}$, $x=\frac{\beta+4(q-1)}{4-\beta}=\frac{8q^2-4q-1}{6q+1}$. Equivalently, we choose $a=\frac{6q+1}{8q^2-4q-1}$ and $b=-\frac{4q+2}{4q+1}(a+1)=-\frac{4q(2q+1)}{8q^2-4q-1}$. Then \cref{eq. Final ineq.} turns out to be
    \begin{align}\label{eq. Final (b)}
        &\left(2(q-1)\lambda-1\right)
        \int_{\partial M}f^b|\nabla f|^2
        +\left(\frac{2q(2q+1)}{6q+1}\lambda-(n-1)\right)
        \int_{\partial M}f^b\chi^2
        -\frac{2q(2q+1)}{6q+1}\int_{\partial M}f^{b+a-aq}\chi^2\notag\\
        \geq & \int_M v^b\left(
        \left|\nabla^2 v-\frac{\Delta v}{n}g\right|^2
        +\mathrm{Ric}(\nabla v,\nabla v)
        \right)\notag\\
        &+\left(\frac{6q+1}{8q^2-4q-1}\right)^2
        \left(\frac{2q(4q+1)}{6q+1}\right)
        \left(
        \frac{-4(n-2)q^2+2(n+1)q+n}{(6q+1)n}
        \right)
        \int_M v^{b-2}|\nabla v|^4.
    \end{align}
    Then the condition (2) in \cref{thm. C. Ricci} implies that the left hand side of \cref{eq. Final (b)} is less or equal to zero, while the right hand side of \cref{eq. Final (b)} is larger or equal to zero. As before, we could conclude that $u$ is a constant function on $M$.
\end{proof}

 \section{Proof of \cref{thm. GL25}}\label{sec3}
 We shall first establish some identities that holds on general Riemannian manifolds.
\begin{lemma}\label{cor.div 1}
    Let $(M^n,g)$ be a Riemannian manifold. For any $v\in C^{\infty}(M)$ and constant $d>0$, there holds
    \begin{align}\label{eq. div 1}
        \mathrm{div}\left(\nabla_{\nabla v}\nabla v-\frac{\Delta v}{d}\nabla v\right)
        =\left|\nabla ^2 v-\frac{\Delta v}{n}g\right|^2
        +\left(\frac{1}{n}-\frac{1}{d}\right)(\Delta v)^2
        +\left(1-\frac{1}{d}\right)
        \langle\nabla \Delta v,\nabla v\rangle+\mathrm{Ric}(\nabla v,\nabla v).
    \end{align}
\end{lemma}
\begin{proof}
    This is a straightforward corollary of Bochner formula.
\end{proof}
\begin{lemma}\label{cor. div 1w}
    Let $(M^n,g)$ be a Riemannian manifold admitting a smooth function $w$ with $\nabla ^2 w=\frac{\Delta w}{n}g$, then for any $v\in C^{\infty}(M)$ and constant $d>0$ there holds
    \begin{align}\label{eq. div 1w}
        \mathrm{div}\left(\nabla_{\nabla w}\nabla v-\frac{\Delta v}{d}\nabla w\right)
        =\left(\frac{1}{n}-\frac{1}{d}\right)\Delta v\Delta w
        -\left(1-\frac{1}{n}\right)\langle\nabla \Delta w,\nabla v\rangle
        +\left(1-\frac{1}{d}\right)\langle\nabla \Delta v,\nabla w\rangle.
    \end{align}
\end{lemma}
\begin{proof}
    Notice that
    \begin{align*}
        \mathrm{div}(\nabla_{\nabla w}\nabla v)
        &=\frac{1}{2}\langle \nabla^2 v,\mathcal{L}_{\nabla w}g\rangle
        +\langle\nabla\Delta v,\nabla w\rangle
        +Ric(\nabla v,\nabla w)\\
        &=\frac{1}{n}\Delta v\Delta w
        +\langle\nabla\Delta v,\nabla w\rangle
        +Ric(\nabla v,\nabla w),
    \end{align*}
    where we used Ricci's identity
    \begin{align*}
        \mathrm{div}(\nabla^2 v)=\nabla \Delta v+\mathrm{Ric}(\nabla v,\cdot).
    \end{align*}
    The hypothesis on $w$ implies that the curvature operator $R$ satisfies
    \begin{align*}
        R(X,Y)\nabla w=\left\langle X, \frac{\nabla \Delta w}{n}\right\rangle Y
        -\left\langle Y, \frac{\nabla \Delta w}{n}\right\rangle X,\quad \forall X,Y\in T_p(M).
    \end{align*}
    It follows that 
    \begin{align*}
        Ric(\nabla w,\cdot)=-\frac{n-1}{n}\nabla \Delta w.
    \end{align*}
    Therefore, we derive
    \begin{align*}
        \mathrm{div}(\nabla_{\nabla w}\nabla v)
        =\frac{1}{n}\Delta v\Delta w
        +\langle\nabla\Delta v,\nabla w\rangle
        -\frac{n-1}{n}\langle\nabla \Delta w,\nabla v\rangle.
    \end{align*}
    Hence the desired identity follows.
\end{proof}
Now we are ready to present our proof of \cref{thm. GL25}.
Roughly speaking, we regard the equation \cref{eq. harmonic} as the  critical power case in a $d-$dimensional space, where \begin{align}\label{eq. intrinsic dimension}
    d:=\frac{2q}{q-1}\geq n.
\end{align}
Then we modify Escobar's argument \cite{Esc90} to fit in this "critical case", and use the boundary condition to tackle the emerging terms in this case.

At the end of our proof, we shall review and compare our choice of parameters with those in Gu-Li \cite{GL25}.

\begin{proof}[Proof of \cref{thm. GL25}:] Denote $g$ to be the Euclidean metric on $\mathbb{B}^n$ and $\nu$ to be the unit outer normal vector of $\partial \mathbb{B}^n=\mathbb{S}^{n-1}$.

\textbf{Step 1:}    Consider the intrinsic dimension $d:=\frac{2q}{q-1}$ so that $q=\frac{d}{d-2}$. Then $d\geq n\Leftrightarrow q\leq \frac{n}{n-2}$. Let $v:=u^{-\frac{2}{d-2}}$, $\chi:=\frac{\partial v}{\partial \nu}, f:=v|_{\mathbb{S}^{n-1}}$, then $v$ satisfies
    \begin{align}\label{eq. PDE of v}
        \begin{cases}
            \Delta v=\frac{d}{2}v^{-1}|\nabla v|^2 &\mathrm{in}\ \mathbb{B}^n,\\
            \chi=\frac{2}{d-2}(\lambda f-1) &\mathrm{on}\ \mathbb{S}^{n-1}.
        \end{cases}
    \end{align}
    We shall consider the P-function $P:=v^{-1}|\nabla v|^2=\frac{2}{d}\Delta v$, which satisfies $P\equiv constant$ if $u$ is the model solution \cref{eq. model solution}, and derive the equation satisfied by $P$ (i.e. \cref{eq. div 1_1-d}).
    
    It's straightforward to see from \cref{eq. PDE of v} that
    \begin{align}\label{eq. nabla Delta v}
        \nabla \Delta v
        =dv^{-1}
        \left(\nabla_{\nabla v}\nabla v-\frac{\Delta v}{d}\nabla v\right).
    \end{align}
    Hence by \cref{cor.div 1} we have
    \begin{align}\label{eq. div 1_1-d}
        \mathrm{div}\left(
        v^{1-d}\left(\nabla_{\nabla v}\nabla v-\frac{\Delta v}{d}\nabla v\right)
        \right)
        =v^{1-d}\left(
        \left|\nabla^2 v-\frac{\Delta v}{n}g\right|^2
        +\left(\frac{1}{n}-\frac{1}{d}\right)(\Delta v)^2
        \right).
    \end{align}
    Integrate this equation over $\mathbb{B}^n$ and use the boundary value condition \cref{eq. PDE of v}, we get
    \begin{align}\label{eq. 1}
        &\int_{\mathbb{B}^n}v^{1-d}\left(
        \left|\nabla^2 v-\frac{\Delta v}{n}g\right|^2
        +\left(\frac{1}{n}-\frac{1}{d}\right)(\Delta v)^2
        \right)\notag\\
        =&\int_{\mathbb{S}^{n-1}}f^{1-d}
        \left(\nabla^2 v-\frac{\Delta v}{d}g\right)(\nabla v,\nu)\notag
        =\int_{\mathbb{S}^{n-1}}f^{1-d}
        \left(\nabla^2 v-\frac{\Delta v}{d}g\right)(\nabla f+\chi\nu,\nu)\\
        =&\int_{\mathbb{S}^{n-1}}f^{1-d}(\langle \nabla f,\nabla \chi\rangle-|\nabla f|^2)
        +\int_{\mathbb{S}^{n-1}}\chi f^{1-d}
        \left(\nabla^2 v-\frac{\Delta v}{d}g\right)(\nu,\nu)\notag\\
        =&\left(\frac{2}{d-2}\lambda-1\right)
        \int_{\mathbb{S}^{n-1}}f^{1-d}|\nabla f|^2
        +\frac{2}{d-2}\lambda
        \int_{\mathbb{S}^{n-1}}f^{2-d}
        \left(\nabla^2 v-\frac{\Delta v}{d}g\right)(\nu,\nu)\notag\\
        &-\frac{2}{d-2}
        \int_{\mathbb{S}^{n-1}}f^{1-d}
        \left(\nabla^2 v-\frac{\Delta v}{d}g\right)(\nu,\nu).
    \end{align}

    \textbf{Step 2:} Take $w:=\frac{1-|x|^2}{2}$. Then $w$ satisfies: $w|_{\mathbb{S}^{n-1}}=0,\ \nabla w|_{\mathbb{S}^{n-1}}=-\nu,\  \nabla^2 w=\frac{\Delta w}{n}g,\ \Delta w\equiv-n$. We shall use $w$ to tackle the last two terms in the right hand side of \cref{eq. 1}.

    Now use \cref{cor. div 1w}, \cref{eq. nabla Delta v} and $\Delta w\equiv-n$, we derive
    \begin{align}\label{eq. div 1w_2-d}
        \mathrm{div}\left(
        v^{2-d}
        \left(\nabla_{\nabla w}\nabla v-\frac{\Delta v}{d}\nabla w\right)
        \right)
        =\left(\frac{n}{d}-1\right)v^{2-d}\Delta v
        +v^{1-d}
        \left\langle\nabla_{\nabla v}\nabla v-\frac{\Delta v}{d}\nabla v,\nabla w\right\rangle.
    \end{align}
    It follows from \cref{eq. div 1_1-d} and \cref{eq. div 1w_2-d} that
    \begin{align*}
        &wv^{1-d}\left(
        \left|\nabla^2 v-\frac{\Delta v}{n}g\right|^2
        +\left(\frac{1}{n}-\frac{1}{d}\right)(\Delta v)^2
        \right)
        =w\mathrm{div}\left(
        v^{1-d}\left(\nabla_{\nabla v}\nabla v-\frac{\Delta v}{d}\nabla v\right)
        \right)\notag\\
        =&\mathrm{div}\left(
        wv^{1-d}
        \left(\nabla_{\nabla v}\nabla v-\frac{\Delta v}{d}\nabla v\right)
        \right)
        -v^{1-d}\left\langle\nabla_{\nabla v}\nabla v-\frac{\Delta v}{d}\nabla v,\nabla w\right\rangle\notag\\
        =&\mathrm{div}\left(
        wv^{1-d}\left(\nabla_{\nabla v}\nabla v-\frac{\Delta v}{d}\nabla v\right)
        \right)
        -\mathrm{div}\left(
        v^{2-d}
        \left(\nabla_{\nabla w}\nabla v-\frac{\Delta v}{d}\nabla w\right)
        \right)
        +\left(\frac{n}{d}-1\right)v^{2-d}\Delta v.
    \end{align*}
    Integrate this equation over $\mathbb{B}^n$ and notice that $\nabla w|_{\mathbb{S}^{n-1}}=-\nu$, we get
    \begin{align}\label{eq. 1w_2-d}
        &\int_{\mathbb{B}^n}wv^{1-d}\left(
        \left|\nabla^2 v-\frac{\Delta v}{n}g\right|^2
        +\left(\frac{1}{n}-\frac{1}{d}\right)(\Delta v)^2
        \right)\notag\\
        =&\int_{\mathbb{S}^{n-1}}f^{2-d}\left(\nabla^2 v-\frac{\Delta v}{d}g\right)(\nu,\nu)
        +\left(\frac{n}{d}-1\right)
        \int_{\mathbb{B}^n}v^{2-d}\Delta v.
    \end{align}

    Once again,\cref{eq. nabla Delta v} and \cref{cor. div 1w} implies
    \begin{align*}
        \mathrm{div}\left(
        v^{1-d}
        \left(\nabla_{\nabla w}\nabla v-\frac{\Delta v}{d}\nabla w\right)
        \right)
        =\left(\frac{n}{d}-1\right)v^{1-d}\Delta v.
    \end{align*}
    Integrate this equation over $\mathbb{B}^n$ and notice that $\nabla w|_{\mathbb{S}^{n-1}}=-\nu$, we get
    \begin{align}\label{eq. 1w_1-d}
        \int_{\mathbb{S}^{n-1}}f^{1-d}
        \left(\nabla^2 v-\frac{\Delta v}{d}g\right)(\nu,\nu)
        =\left(1-\frac{n}{d}\right)
        \int_{\mathbb{B}^n}v^{1-d}\Delta v.
    \end{align}
    Note that \cref{eq. 1w_1-d} is exactly the Pohozaev identity \cite[Proposition 1.4]{Sch88} used by Escobar \cite{Esc90}, whose validity comes from the fact that $\nabla w$ is a closed conformal vector field.

    Plug \cref{eq. 1w_2-d} and \cref{eq. 1w_1-d} into \cref{eq. 1}, we derive the key integral identity
    \begin{align*}
        &\int_{\mathbb{B}^n}
        \left(1-\frac{2}{d-2}\lambda w\right)
        v^{1-d}\left(
        \left|\nabla^2 v-\frac{\Delta v}{n}g\right|^2
        +\left(\frac{1}{n}-\frac{1}{d}\right)(\Delta v)^2
        \right)\notag\\
        =&\left(\frac{2}{d-2}\lambda-1\right)
        \int_{\mathbb{S}^{n-1}}f^{1-d}|\nabla f|^2
        +\frac{2}{d-2}
        \left(\frac{n}{d}-1\right)\left(
        \int_{\mathbb{B}^n}v^{1-d}\Delta v-\lambda \int_{\mathbb{B}^n}v^{2-d}\Delta v
        \right).
    \end{align*}

    Notice that $1-\frac{2}{d-2}\lambda w\geq1-\frac{1}{d-2}\lambda\geq 0$ , $\frac{2}{d-2}\lambda-1\leq 0$ and $\frac{n}{d}-1\leq 0$, our proof would be complete once we show that
    \begin{align}\label{eq. goal}
        \int_{\mathbb{B}^n}v^{1-d}\Delta v-\lambda \int_{\mathbb{B}^n}v^{2-d}\Delta v\geq 0.
    \end{align}

    \textbf{Step 3:} We shall use the boundary condition \cref{eq. PDE of v} to make \cref{eq. goal} homogeneous.
    
    The boundary  condition in \cref{eq. PDE of v} implies
    \begin{align*}
        \int_{\mathbb{S}^{n-1}}f^{2-d}\chi
        =\int_{\mathbb{B}^n}\mathrm{div}(v^{2-d}\nabla v)
        =\int_{\mathbb{B}^n}(v^{2-d}\Delta v+(2-d)v^{-d}|\nabla v|^2)
        =\frac{4-d}{d}\int_{\mathbb{B}^n}v^{2-d}\Delta v,\\
        \int_{\mathbb{S}^{n-1}}f^{1-d}\chi
        =\int_{\mathbb{B}^n}\mathrm{div}(v^{1-d}\nabla v)
        =\int_{\mathbb{B}^n}(v^{1-d}\Delta v+(1-d)v^{-d}|\nabla v|^2)
        =\frac{2-d}{d}\int_{\mathbb{B}^n}v^{1-d}\Delta v.
    \end{align*}
    Therefore we have
    \begin{align*}
        \int_{\mathbb{S}^{n-1}}f^{1-d}\chi^2
        =\frac{2}{d-2}\lambda\int_{\mathbb{S}^{n-1}}f^{2-d}\chi
        -\frac{2}{d-2}\int_{\mathbb{S}^{n-1}}f^{1-d}\chi
        =\frac{2(4-d)}{d(d-2)}\lambda\int_{\mathbb{B}^n} v^{2-d}\Delta v+\frac{2}{d}\int_{\mathbb{B}^n}v^{1-d}\Delta v.
    \end{align*}
    Equivalently,
    \begin{align*}
        \int_{\mathbb{B}^n}v^{1-d}\Delta v
        =\frac{d}{2}\int_{\mathbb{S}^{n-1}}f^{1-d}\chi^2
        +\frac{d-4}{d-2}\lambda\int_{\mathbb{B}^n}v^{2-d}\Delta v.
    \end{align*}
    Hence we obtain
    \begin{align*}
        \int_{\mathbb{B}^n}v^{1-d}\Delta v-\lambda \int_{\mathbb{B}^n}v^{2-d}\Delta v
        =\frac{d}{2}\int_{\mathbb{S}^{n-1}}f^{1-d}\chi^2-\frac{2}{d-2}\lambda\int_{\mathbb{B}^n}v^{2-d}\Delta v.
    \end{align*}
Therefore, it suffices to show that $\frac{d}{2}\int_{\mathbb{S}^{n-1}}f^{1-d}\chi^2-\frac{2}{d-2}\lambda\int_{\mathbb{B}^n}v^{2-d}\Delta v\geq 0$.

    \textbf{Step 4:} Using $\nabla^2 w=-g$ and \cref{eq. PDE of v}, we could express $\int_{\mathbb{S}^{n-1}}f^{1-d}\chi^2$ as
    \begin{align*}
        &\int_{\mathbb{S}^{n-1}}f^{1-d}\chi^2
        =-\int_{\mathbb{B}^n}\mathrm{div}\left(v^{1-d}\langle \nabla v,\nabla w\rangle\nabla v\right)\notag\\
        =&-\int_{\mathbb{B}^n}\left\{v^{1-d}\langle \nabla v,\nabla w\rangle\Delta v+(1-d)v^{-d}\langle \nabla v,\nabla w\rangle|\nabla v|^2
        +v^{1-d}\langle\nabla_{\nabla v}\nabla v,\nabla w\rangle
        +v^{1-d}\nabla^2 w(\nabla v,\nabla v)\right\}\notag\\
        =&-\int_{\mathbb{B}^n}v^{1-d}
        \left\langle\nabla_{\nabla v}\nabla v-\frac{d-2}{d}\Delta v\nabla v,\nabla w\right\rangle
        +\frac{2}{d}\int_{\mathbb{B}^n}v^{2-d}\Delta v\notag\\
        =&\int_{\mathbb{B}^n}\left\{w\mathrm{div}\left(
        v^{1-d}
        \left(\nabla_{\nabla v}\nabla v-\frac{d-2}{d}\Delta v\nabla v\right)
        \right)
        -\mathrm{div}\left(
        wv^{1-d}
        \left(\nabla_{\nabla v}\nabla v-\frac{d-2}{d}\Delta v\nabla v\right)
        \right)\right\}\\
        &+\frac{2}{d}\int_{\mathbb{B}^n}v^{2-d}\Delta v.
    \end{align*}
    Therefore, 
    \begin{align}\label{eq. last equation}
        &\frac{d}{2}\int_{\mathbb{S}^{n-1}}f^{1-d}\chi^2
        -\frac{2}{d-2}\lambda\int_{\mathbb{B}^n}v^{2-d}\Delta v\notag \\
        =&\left(1-\frac{2}{d-2}\lambda\right)
        \int_{\mathbb{B}^n}v^{2-d}\Delta v
        +\frac{d}{2}\int_{\mathbb{B}^n}w\mathrm{div}\left(
        v^{1-d}
        \left(\nabla_{\nabla v}\nabla v-\frac{d-2}{d}\Delta v\nabla v\right)
        \right).
    \end{align}
    Finally, we use \cref{cor.div 1} and \cref{eq. nabla Delta v} to calculate the last term in \cref{eq. last equation} as follows:
    \begin{align*}
        &w\mathrm{div}\left(
        v^{1-d}
        \left(\nabla_{\nabla v}\nabla v-\frac{d-2}{d}\Delta v\nabla v\right)
        \right)\\
        =&wv^{1-d}\left(
            \mathrm{div}
            \left(\nabla_{\nabla v}\nabla v-\frac{d-2}{d}\Delta v\nabla v\right)
            +(1-d)v^{-1}
            \left\langle \nabla_{\nabla v}\nabla v-\frac{d-2}{d}\Delta v\nabla v,\nabla v\right\rangle
        \right)\\
        =&wv^{1-d}\left(
        \left|\nabla^2 v-\frac{\Delta v}{n}g\right|^2
        +\left(\frac{1}{n}-\frac{d-2}{d}\right)
        (\Delta v)^2
        +2v^{-1}
        \left\langle \nabla_{\nabla v}\nabla v-\frac{\Delta v}{d}\nabla v,\nabla v\right\rangle\right.\\
        &\quad\quad\quad\left.
        +(1-d)v^{-1}\left\langle \nabla_{\nabla v}\nabla v-\frac{d-2}{d}\Delta v\nabla v,\nabla v\right\rangle
        \right)\\
        =&wv^{1-d}\left(
        \left|\nabla^2 v-\frac{\Delta v}{n}g\right|^2
        +(3-d)v^{-1}\langle \nabla_{\nabla v}\nabla v,\nabla v\rangle
        +\left(\frac{1}{n}-\frac{d-2}{d}+\frac{2(d-3)}{d}\right)(\Delta v)^2
        \right)\\
        =&wv^{1-d}\left(
        \left|\nabla^2 v-\frac{\Delta v}{n}g+\frac{3-d}{2}\frac{\mathrm{d}v\otimes\mathrm{d}v}{v}\right|^2
        +\frac{(d-3)(d+2n)+3(d-n)}{nd^2}(\Delta v)^2
        \right)\geq 0.
    \end{align*}
    This finishes the proof of \cref{thm. GL25}.
\end{proof}
\begin{remark}
    Gu-Li chose the power of $v$ as $a=-\frac{q+1}{q-1}$ (see \cite[Section 3.1]{GL25}), which is exactly $1-d$ in our \cref{eq. div 1_1-d}. They chose the combination coefficient of the vector field as $b=\frac{q-1}{2q}$ (see \cite[Section 3.2]{GL25}), which is exactly $\frac{1}{d}$ in our \cref{eq. nabla Delta v}. They finally choose the combination coefficient of the weight function as $\frac{2}{1+c}=\lambda(q-1)$ (see \cite[Section 3.3]{GL25}), which equals $\frac{2}{d-2}\lambda$ and appears naturally in our \cref{eq. 1}.
\end{remark}
\section{Proof of \cref{thm. C.}}\label{sec4}
\begin{proof}[Proof of \cref{thm. C.}:]
    Denote $g$ to be the Euclidean metric on $\mathbb{B}^n$ and $\nu$ to be the unit outer normal vector of $\partial \mathbb{B}^n=\mathbb{S}^{n-1}$.

    Assume $u$ is not identically zero, then $u$ is superharmonic by \cref{eq. general equation of u-p}. Hence the maximum principle and the Hopf lemma imply that $u>0$ on $\overline{\mathbb{B}^n}$. 
    
    The remaining calculation is a minor modification of our proof of \cref{thm. GL25} in \cref{sec3}. For completeness, we repeat the calculation and provide all details.

\textbf{Step 1:}    Define $d:=\frac{2(p+1)}{p-1}$. Then $d\geq n\Leftrightarrow p\leq \frac{n+2}{n-2}$. Let $v:=u^{-\frac{2}{d-2}}$, $\chi:=\frac{\partial v}{\partial \nu}, f:=v|_{\mathbb{S}^{n-1}}$, then $v$ satisfies
    \begin{align}\label{eq. PDE of v_nonlinear}
        \begin{cases}
            \Delta v=\frac{d}{2}v^{-1}|\nabla v|^2+\frac{n-2}{2(d-2)}\frac{R}{n-1}v^{-1} &\mathrm{in}\ \mathbb{B}^n,\\
            \chi=\frac{2}{d-2}\lambda f-\frac{n-2}{d-2}\frac{H}{n-1} &\mathrm{on}\ \mathbb{S}^{n-1}.
        \end{cases}
    \end{align}
    The key point is that \cref{eq. nabla Delta v} still holds in this case for a solution $v
    $ in \cref{eq. PDE of v_nonlinear}:
    \begin{align}\label{eq. 2nd nabla delta v}
        \nabla \Delta v
        =dv^{-1}
        \left(\nabla_{\nabla v}\nabla v-\frac{\Delta v}{d}\nabla v\right).
    \end{align}

    Hence by \cref{cor.div 1} we get \cref{eq. div 1_1-d}. Integrate this over $\mathbb{B}^n$ and use the boundary condition \cref{eq. PDE of v_nonlinear}, we get
    \begin{align}\label{eq. 1_p}
        &\int_{\mathbb{B}^n}v^{1-d}\left(
        \left|\nabla^2 v-\frac{\Delta v}{n}g\right|^2
        +\left(\frac{1}{n}-\frac{1}{d}\right)(\Delta v)^2
        \right)\notag\\
        =&\int_{\mathbb{S}^{n-1}}f^{1-d}
        \left(\nabla^2 v-\frac{\Delta v}{d}g\right)(\nabla v,\nu)\notag
        =\int_{\mathbb{S}^{n-1}}f^{1-d}
        \left(\nabla^2 v-\frac{\Delta v}{d}g\right)(\nabla f+\chi\nu,\nu)\\
        =&\int_{\mathbb{S}^{n-1}}f^{1-d}(\langle \nabla f,\nabla \chi\rangle-|\nabla f|^2)
        +\int_{\mathbb{S}^{n-1}}\chi f^{1-d}
        \left(\nabla^2 v-\frac{\Delta v}{d}g\right)(\nu,\nu)\notag\\
        =&\left(\frac{2}{d-2}\lambda-1\right)
        \int_{\mathbb{S}^{n-1}}f^{1-d}|\nabla f|^2
        +\frac{2}{d-2}\lambda
        \int_{\mathbb{S}^{n-1}}f^{2-d}
        \left(\nabla^2 v-\frac{\Delta v}{d}g\right)(\nu,\nu)\notag\\
        &-\frac{n-2}{d-2}\frac{H}{n-1}
        \int_{\mathbb{S}^{n-1}}f^{1-d}
        \left(\nabla^2 v-\frac{\Delta v}{d}g\right)(\nu,\nu).
    \end{align}
    
\textbf{Step 2:} Take $w:=\frac{1-|x|^2}{2}$. Then $w$ satisfies: $w|_{\mathbb{S}^{n-1}}=0,\ \nabla w|_{\mathbb{S}^{n-1}}=-\nu,\  \nabla^2 w=\frac{\Delta w}{n}g,\ \Delta w\equiv-n$. We shall use $w$ to tackle the last two terms in the right hand side of \cref{eq. 1_p}

Now use \cref{cor. div 1w}, \cref{eq. 2nd nabla delta v} and $\Delta w\equiv-n$, we derive \cref{eq. div 1w_2-d}. It follows from \cref{eq. div 1_1-d} and \cref{eq. div 1w_2-d} that
\begin{align*}
        &wv^{1-d}\left(
        \left|\nabla^2 v-\frac{\Delta v}{n}g\right|^2
        +\left(\frac{1}{n}-\frac{1}{d}\right)(\Delta v)^2
        \right)
        =w\mathrm{div}\left(
        v^{1-d}
        \left(\nabla_{\nabla v}\nabla v-\frac{\Delta v}{d}\nabla v\right)
        \right)\notag\\
        =&\mathrm{div}\left(
        wv^{1-d}
        \left(\nabla_{\nabla v}\nabla v-\frac{\Delta v}{d}\nabla v\right)
        \right)
        -v^{1-d}
        \left\langle\nabla_{\nabla v}\nabla v-\frac{\Delta v}{d}\nabla v,\nabla w\right\rangle\notag\\
        =&\mathrm{div}\left(
        wv^{1-d}
        \left(\nabla_{\nabla v}\nabla v-\frac{\Delta v}{d}\nabla v\right)
        \right)
        -\mathrm{div}\left(
        v^{2-d}
        \left(\nabla_{\nabla w}\nabla v-\frac{\Delta v}{d}\nabla w\right)
        \right)
        +\left(\frac{n}{d}-1\right)v^{2-d}\Delta v.
    \end{align*}
    Integrate this equation over $\mathbb{B}^n$ and notice that $\nabla w|_{\mathbb{S}^{n-1}}=-\nu$, we get \cref{eq. 1w_2-d}:
    \begin{align*}
        &\int_{\mathbb{B}^n}wv^{1-d}\left(
        \left|\nabla^2 v-\frac{\Delta v}{n}g\right|^2
        +\left(\frac{1}{n}-\frac{1}{d}\right)(\Delta v)^2
        \right)\\
        =&\int_{\mathbb{S}^{n-1}}f^{2-d}
        \left(\nabla^2 v-\frac{\Delta v}{d}g\right)(\nu,\nu)
        +\left(\frac{n}{d}-1\right)
        \int_{\mathbb{B}^n}v^{2-d}\Delta v.
    \end{align*}

    Once again, \cref{eq. 2nd nabla delta v} and \cref{cor. div 1w} implies
    \begin{align*}
        \mathrm{div}\left(
        v^{1-d}
        \left(\nabla_{\nabla w}\nabla v-\frac{\Delta v}{d}\nabla w\right)
        \right)
        =\left(\frac{n}{d}-1\right)v^{1-d}\Delta v.
    \end{align*}
    Integrate this equation over $\mathbb{B}^n$ and notice that $\nabla w|_{\mathbb{S}^{n-1}}=-\nu$, we get \cref{eq. 1w_1-d}:
    \begin{align*}
        \int_{\mathbb{S}^{n-1}}f^{1-d}
        \left(\nabla^2 v-\frac{\Delta v}{d}g\right)(\nu,\nu)
        =\left(1-\frac{n}{d}\right)
        \int_{\mathbb{B}^n}v^{1-d}\Delta v.
    \end{align*}

    Plug \cref{eq. 1w_2-d} and \cref{eq. 1w_1-d} into \cref{eq. 1_p}, we derive the following key integral identity:
    \begin{align}\label{eq. final integral equation}
    0\leq&\int_{\mathbb{B}^n}
    \left(1-\frac{2}{d-2}\lambda w\right)
    v^{1-d}\left(
        \left|\nabla^2 v-\frac{\Delta v}{n}g\right|^2
        +\left(\frac{1}{n}-\frac{1}{d}\right)(\Delta v)^2
        \right)\notag\\
        =&\left(\frac{2}{d-2}\lambda-1\right)\int_{\mathbb{S}^{n-1}}f^{1-d}|\nabla f|^2
        +\frac{2}{d-2}\left(\frac{n}{d}-1\right)\left(
        \frac{n-2}{2}\frac{H}{n-1}\int_{\mathbb{B}^n}v^{1-d}\Delta v-\lambda \int_{\mathbb{B}^n}v^{2-d}\Delta v
        \right).
\end{align}

\textbf{Step 3:} Now we verify the first case in \cref{thm. C.}.

If $1<p\leq \frac{n+2}{n-2}$ and $\lambda=0$, then \cref{eq. final integral equation} reduces to 
\begin{align*}
    0\leq& \int_{\mathbb{B}^n}v^{1-d}\left(
        \left|\nabla^2 v-\frac{\Delta v}{n}g\right|^2
        +\left(\frac{1}{n}-\frac{1}{d}\right)(\Delta v)^2
        \right)\notag\\
        =&-\int_{\mathbb{S}^{n-1}}f^{1-d}|\nabla f|^2
        +\frac{n-2}{d-2}\frac{H}{n-1}
        \left(\frac{n}{d}-1\right)
        \int_{\mathbb{B}^n}v^{1-d}\Delta v\\
        \leq&0.
\end{align*}
This forces 
\begin{align*}
    \nabla^2v=\frac{\Delta v}{n}g \quad\mathrm{in} \ \mathbb{B}^n,
\end{align*}
and $v|_{\mathbb{S}^{n-1}}=constant$. It follows from Ricci's identity that
\begin{align*}
    \nabla \Delta v
    =\mathrm{div}(\nabla^2v)
    =\mathrm{div}\left(\frac{\Delta v}{n}g\right)
    =\frac{\nabla \Delta v}{n}.
\end{align*}
Therefore, $\Delta v$ is a constant in $\mathbb{B}^n$. Combining with $v|_{\mathbb{S}^{n-1}}=constant$, we could set $v(x)=r|x|^2+s$ for some constants $r,s$. The equation \cref{eq. PDE of v_nonlinear} reduces to
\begin{align*}
    \begin{cases}
        2nr^2|x|^2+2nrs=2dr^2|x|^2+\frac{n-2}{2(d-2)}\frac{R}{n-1} &\mathrm{in}\ \mathbb{B}^n, \\
        2r=-\frac{n-2}{d-2}\frac{H}{n-1} &\mathrm{on}\ \mathbb{S}^{n-1}.
    \end{cases}
\end{align*}
If follows that $d=n$ and 
\begin{align*}
\begin{cases}
    rs=\frac{1}{4n}\frac{R}{n-1},\\
    r=-\frac{1}{2}\frac{H}{n-1}.
\end{cases}
\end{align*}
If $H=0$, then $r=0$ and $R=0$. It's a contradiction.

If $H>0$, then $r=-\frac{1}{2}\frac{H}{n-1}<0, s=-\frac{1}{2n}\frac{R}{H}<0$. This contradicts with the fact that $v>0$ on $\mathbb{B}^n$.

In conclusion, there is no positive solution in this case, and the only nonnegative solution of \cref{eq. general equation of u-p} is $u\equiv 0$.

\textbf{Step 4:} If $p=\frac{n+2}{n-2}$ and  $0<\lambda< \frac{2}{p-1}=\frac{n-2}{2}$, then \cref{eq. final integral equation} reduces to
\begin{align*}
    0\leq&\int_{\mathbb{B}^n}
    \left(1-\frac{2}{n-2}\lambda w\right)
    v^{1-n}
        \left|\nabla^2 v-\frac{\Delta v}{n}g\right|^2
        \notag
        =\left(\frac{2}{n-2}\lambda-1\right)
        \int_{\mathbb{S}^{n-1}}f^{1-n}|\nabla f|^2\leq 0.
\end{align*}
This forces
\begin{align*}
    \nabla^2v=\frac{\Delta v}{n}g \quad\mathrm{in} \ \mathbb{B}^n,
\end{align*}
and $v|_{\mathbb{S}^{n-1}}=constant$. As before, one could derive that $\Delta v$ is a constant in $\mathbb{B}^n$ and set $v(x)=r|x|^2+s$ for some constants $r,s$. Then it follows from the equation \cref{eq. PDE of v_nonlinear} that
\begin{align*}
    \begin{cases}
        4rs=\frac{R}{n(n-1)},\\
        r=\frac{\lambda}{n-2}(r+s)-\frac{H}{2(n-1)}.
    \end{cases}
\end{align*}
Since $v$ is positive, we have $s>0$ and we solve
\begin{align*}
    r=\frac{1}{2\epsilon}\sqrt{\frac{R}{n(n-1)}},\quad 
    s=\frac{\epsilon}{2}\sqrt{\frac{R}{n(n-1)}},
\end{align*}
where $0<\epsilon=\sqrt{\frac{n(n-1)}{R}}
            \left(
        \frac{n-2}{2\lambda}\frac{H}{n-1}+\sqrt{(\frac{n-2}{2\lambda})^2(\frac{H}{n-1})^2+\frac{R}{n(n-1)}(\frac{n-2}{\lambda}-1)}
            \right)$.

In conclusion, $v(x)=\sqrt{\frac{R}{n(n-1)}}
\left(\frac{1}{2\epsilon}|x|^2+\frac{\epsilon}{2}\right)$ and $u=v^{-\frac{n-2}{2}}$ is the desired solution.

\textbf{Step 5:} If $p=\frac{n+2}{n-2}$ and $\lambda=\frac{n-2}{2}$, then \cref{eq. final integral equation} reduces to
    \begin{align*}
        0\leq&\int_{\mathbb{B}^n}
        \left(1-\frac{2}{n-2}\lambda w\right)
        v^{1-n}
        \left|\nabla^2 v-\frac{\Delta v}{n}g\right|^2
        \notag
        =\left(\frac{2}{n-2}\lambda-1\right)
        \int_{\mathbb{S}^{n-1}}f^{1-n}|\nabla f|^2= 0.
    \end{align*}
    As before, we could derive that $\Delta v$ is a constant in $\mathbb{B}^n$ (but $v|_{\mathbb{S}^{n-1}}$ is not necessarily constant) and set $v(x)=r|x|^2+\langle\xi,x\rangle+s$, where $r,s$ are constants and $\xi\in\mathbb{R}^n$. It follows from the equation \cref{eq. PDE of v_nonlinear} that
    \begin{align*}
        \begin{cases}
            4rs=|\xi|^2+\frac{R}{n(n-1)},\\
            r=s-\frac{H}{n-1}.
        \end{cases}
    \end{align*}
    Since $v$ is positive, we have $s>0$ and we solve
    \begin{align*}
    \begin{cases}
        r=\frac{1}{2}\left(
        \sqrt{(\frac{H}{n-1})^2+\frac{R}{n(n-1)}+|\xi|^2}
        -\frac{H}{n-1}
        \right),\\
        s=\frac{1}{2}\left(
        \sqrt{(\frac{H}{n-1})^2+\frac{R}{n(n-1)}+|\xi|^2}
        +\frac{H}{n-1}
        \right).
    \end{cases}
    \end{align*}    
    For the fixed $\xi\in \mathbb{R}^n$, there exists a unique $a\in\mathbb{B}^n$ such that
    \begin{align*}
        \xi=2\sqrt{\left(\frac{H}{n-1}\right)^2+\frac{R}{n(n-1)}}\,\frac{a}{1-|a|^2}.
    \end{align*}
    Hence we derive that
    \begin{align*}
        \begin{cases}
            r=\frac{1}{2}
            \left(
            \frac{1+|a|^2}{1-|a|^2}\sqrt{(\frac{H}{n-1})^2+\frac{R}{n(n-1)}}
            -\frac{H}{n-1}
            \right)
            =\sqrt{\frac{R}{n(n-1)}}\,
            \frac{1+\epsilon^2|a|^2}{2\epsilon(1-|a|^2)},\\
            s=\frac{1}{2}
            \left(
            \frac{1+|a|^2}{1-|a|^2}\sqrt{(\frac{H}{n-1})^2+\frac{R}{n(n-1)}}
            +\frac{H}{n-1}
            \right)
            =\sqrt{\frac{R}{n(n-1)}}\,
            \frac{\epsilon^2+|a|^2}{2\epsilon(1-|a|^2)},\\
            \xi=\sqrt{\frac{R}{n(n-1)}}
            \frac{(1+\epsilon^2)a}{\epsilon(1-|a|^2)},
        \end{cases}     
    \end{align*}
    where $0<\epsilon
            =\sqrt{\frac{n(n-1)}{R}}
            \left(\frac{H}{n-1}+\sqrt{(\frac{H}{n-1})^2+\frac{R}{n(n-1)}}\right)$.
            
In conclusion, $v(x)=\sqrt{\frac{R}{n(n-1)}}
\left(
\frac{1+\epsilon^2|a|^2}{2\epsilon(1-|a|^2)}|x|^2
+\frac{1+\epsilon^2}{\epsilon(1-|a|^2)}\langle x,a\rangle
+\frac{\epsilon^2+|a|^2}{2\epsilon(1-|a|^2)}
\right)$ and $u=v^{-\frac{n-2}{2}}$ is the desired solution.
\end{proof}



\printbibliography[heading=bibintoc]

\appendix
\section{}\label{appendix}

Bidaut-V\'eron and V\'eron \cite[Theorem 6.1]{BVV91} established the following uniqueness result:
\begin{theorem}[\cite{BVV91}]\label{thm.BVV91}
    Let $(M^n,g), n\geq 3$ be a complete, compact Riemannian manifold without boundary.  Assume $Ric\geq n-1$, and $1<q\leq\frac{n+2}{n-2}$, $0<\lambda\leq \frac{n}{q-1}$ are constants. Let $u\in C^2(M)$ be a positive solution of
    \begin{align}\label{eq. Main equaiton}
        -\Delta u+\lambda u=u^{q}\quad \mathrm{in~}M,
    \end{align}
    Then either $u$ is constant on $M$ or $q=\frac{n+2}{n-2},\,\lambda =\frac{n(n-2)}{4}$ and $(M^n,g)$ is isometric to $(\mathbb{S}^n,g_0)$. In the latter case, there holds
    \begin{align*}
        u(x)=\frac{c_n}{(\cosh t+(\sinh t)x\cdot\xi)^{\frac{n-2}{2}}}
    \end{align*}
    for some $\xi\in \mathbb{S}^{n}$ and some constant $t\geq 0$.
\end{theorem}
For the critical power case (i.e. $q=\frac{n+2}{n-2}$), the solution could be classified via a conceptually simple strategy as follows. By carefully analyzing the model solution, Wang \cite[Section 2]{Wan22} come up with an appropriate function $\phi$, known as a P-function, which is constant if  and only if the solution $u$ is given by the model case. Then the Bochner formula implies that, up to a first order term, $\phi$ is a subharmonic function. Hence the maximum principle could be applied to show that $\phi$ must be constant and the proof finishes.

In the subcritical power case (i.e. $1<q<\frac{n+2}{n-2}$), as pointed out by Wang \cite[Section 5]{Wan22}, the choice of parameters (and the P-function) is delicate if one wishes to use the same strategy. In the following, we shall utilize our principle used in our proof of \cref{thm. GL25} to present the appropriate P-function and prove \cref{thm.BVV91} in its full generality.

The idea of the proof is as follows. As in our proof in \cref{sec3}, to handle the subcritical power case, we introduce the intrinsic dimension $d$ such that $q=\frac{d+2}{d-2}$. Then we mimic Wang's proof of the critical power case \cite{Wan22} and finally use an integral inequality (\cref{prop.weak perturbed eigenvalue estimate}) to treat the emerging terms in the subcritical case and conclude the results.

We note that \cref{prop.weak perturbed eigenvalue estimate} is crucial for the \emph{sharp range of $\lambda$} in the subcritical power case. Moreover, it also has its independent interest.
\begin{proposition}\label{prop.weak perturbed eigenvalue estimate}
        Let $(M^n,g)$ be a compact manifold without boundary, satisfying $\mathrm{Ric}\geq (n-1)k\geq0$, then for any constant $s\in(-\infty,0]\cup [\frac{4(n+2)}{4n-1},+\infty)$, there holds 
        \begin{align}\label{weak perturbed eigenvalue estimate}
        &\int_M 
        \left|\nabla^2 v-\frac{\Delta v}{n}g\right|^2 v^{s}
        +\frac{1}{n}\int_M (\Delta v)^2v^{s}
        \geq k\int_M |\nabla v|^2v^{s},\quad \forall\  0<v\in C^{\infty}(M).
    \end{align}
    Moreover, if \cref{weak perturbed eigenvalue estimate} holds for some $s\in (-\infty,0)\cup (\frac{4(n+2)}{4n-1},+\infty)$ and some $0<v\in C^{\infty}(M)$, then $v$ is a constant function.
\end{proposition}
We shall postpone the proof of \cref{prop.weak perturbed eigenvalue estimate} and prove \cref{thm.BVV91} right now.
\begin{proof}[Proof of \cref{thm.BVV91}]
    Let $d:=\frac{2(q+1)}{q-1}$. Then $q\leq\frac{n+2}{n-2}\Leftrightarrow d\geq n$, and $\lambda \leq \frac{n}{q-1}\Leftrightarrow \lambda\leq \frac{n(d-2)}{4}$. Define $v:=u^{-\frac{q-1}{2}}=u^{-\frac{2}{d-2}}$. A straightforward calculation gives
    \begin{align*}
        \Delta v=\frac{d}{2}v^{-1}|\nabla v|^2+\frac{2}{d-2}v^{-1}-\frac{2\lambda}{d-2}v.
    \end{align*}

    Consider the P-function 
    \begin{align*}
        \phi:=\Delta v+\frac{4\lambda}{d-2}v
    =v^{-1}\left(
    \frac{d}{2}|\nabla v|^2+\frac{2}{d-2}+\frac{2\lambda}{d-2}v^2
    \right).
    \end{align*}
    On the one hand, 
    \begin{align*}
    \Delta(v\phi)=v\Delta\phi+2\langle\nabla v,\nabla \phi\rangle+(\Delta v)^2+\frac{4\lambda}{d-2}v\Delta v.
    \end{align*}
    On the other hand, by Bochner formula, there holds
    \begin{align*}
        \Delta(v\phi)&=d
        \left(
        |\nabla^2 v|^2
        +\left\langle\nabla v,\nabla \left(\phi-\frac{4\lambda}{d-2}v\right)\right\rangle+Ric(\nabla v,\nabla v)
        \right)
        +\frac{2\lambda}{d-2}(2|\nabla v|^2+2v\Delta v)\\
        &=d
        \left(
        \left|\nabla^2 v-\frac{\Delta v}{n}g\right|^2+\frac{(\Delta v)^2}{n}
        +\langle\nabla v,\nabla \phi\rangle
        +Ric(\nabla v,\nabla v)
        \right)
        -\frac{4(d-1)\lambda}{d-2}|\nabla v|^2
        +\frac{4\lambda}{d-2}v\Delta v.
    \end{align*}    
    We deduce that $\phi$ satisfies the following equation
    \begin{align*}
        v\Delta \phi+(2-d)\langle\nabla v,\nabla \phi\rangle
        =d\left|\nabla^2 v-\frac{\Delta v}{n}g\right|^2+(\frac{d}{n}-1)(\Delta v)^2+d\ \mathrm{Ric}(\nabla v,\nabla v)-\frac{4(d-1)\lambda}{d-2}|\nabla v|^2.
    \end{align*}
Multiply both sides of this equation by $v^{1-d}$ and use $Ric\geq n-1$,  we derive the following key inequality:
    \begin{align}\label{eq.key inequality}
        \mathrm{div}(v^{2-d}\nabla \phi)
        \geq& (d-n)
        \left|\nabla^2 v-\frac{\Delta v}{n}g\right|^2 v^{1-d}
        +\left(\frac{d}{n}-1\right)(\Delta v)^2v^{1-d}\notag\\
        &+\left(
        d(n-1)-\frac{4(d-1)\lambda}{d-2}
        \right)|\nabla v|^2v^{1-d}.
    \end{align}

    Now integrate \cref{eq.key inequality} over $M$ and use \cref{prop.weak perturbed eigenvalue estimate} (with parameter $s=1-d$), we get
    \begin{align}\label{eq. last step}
        0=\int_M div(v^{2-d}\nabla \phi)
        \geq\left(d-n+d(n-1)-\frac{4(d-1)\lambda}{d-2}\right)\int_M |\nabla v|^2v^{1-d}\geq 0,
    \end{align}
    since $\lambda \leq \frac{n(d-2)}{4}$. Therefore, equality holds in \cref{eq. last step}.
    
    For $1<q<\frac{n+2}{n-2}$, this forces that  the equality holds in \cref{prop.weak perturbed eigenvalue estimate}. Hence $v$ is a constant.

    For $0<\lambda<\frac{n}{q-1}$, \cref{eq. last step} implies that $v$ is a constant.

    For $q=\frac{n+2}{n-2}$ and $\lambda=\frac{n(n-2)}{4}$, we have $\nabla^2 v-\frac{\Delta v}{n}g\equiv 0$ and $\mathrm{Ric}(\nabla v,\nabla v)\equiv n-1$. It follows that $(M^n,g)$ is isometric to the round sphere $\mathbb{S}^n$.
\end{proof}

Finally we present the proof of \cref{prop.weak perturbed eigenvalue estimate} with the help of two auxiliary lemmas. 

\begin{lemma}\label{lem.laplace squared}
        Let $(M^n,g)$ be a compact manifold without boundary. Then for any $u\in C^{\infty}(M)$, there holds
        \begin{align*}
            \int_M(\Delta u)^2
            =&\frac{n}{n-1}\int_M
            \left|\nabla^2 u-\frac{\Delta u}{n}g\right|^2
            +\frac{n}{n-1}\int_M \mathrm{Ric}(\nabla u,\nabla u).
        \end{align*}
    \end{lemma}
    \begin{proof}
        By Bochner formula, 
        \begin{align*}
            \frac{1}{2}\Delta |\nabla u|^2=\left|\nabla^2 u-\frac{\Delta u}{n}g\right|^2
            +\frac{1}{n}(\Delta u)^2
            +\langle\nabla \Delta u,\nabla u\rangle
            +\mathrm{Ric}(\nabla u,\nabla u).
        \end{align*}
        Integrate it over M, we get
        \begin{align*}
           0=\int_M\left(\left|\nabla^2 u-\frac{\Delta u}{n}g\right|^2+\mathrm{Ric}(\nabla u,\nabla u)\right)
            +\frac{1}{n}\int_M(\Delta u)^2
            +\int_M (\mathrm{div}(\Delta u\nabla u)-(\Delta u)^2).
        \end{align*}
        Rearrange it and the proof finishes.
    \end{proof}

    \begin{lemma}\label{lem.grad squared laplace}
        Let $(M^n,g)$ be a compact manifold without boundary. If $u\in C^{\infty}(M)$ is a positive function, then
        \begin{align*}
            \int_M u^{-1}|\nabla u|^2\Delta u
            =&\frac{n}{n+2}\int_M u^{-2}|\nabla u|^4
            -\frac{2n}{n+2}\int_M \left\langle\nabla^2 u-\frac{\Delta u}{n}g,\frac{du\otimes du}{u}-\frac{1}{n}\frac{|\nabla u|^2}{u}g\right\rangle.
        \end{align*}
    \end{lemma}
    \begin{proof}
        By divergence theorem,
        \begin{align*}
            \int_M u^{-1}|\nabla u|^2\Delta u
            =&\int_M\mathrm{div}(u^{-1}|\nabla u|^2\nabla u)-(2u^{-1}\langle\nabla_{\nabla u}\nabla u,\nabla u\rangle-u^{-2}|\nabla u|^4)\\
            =&\int_M u^{-2}|\nabla u|^4
            -2\int_M \left\langle\nabla ^2u,\frac{du\otimes du}{u}-\frac{1}{n}\frac{|\nabla u|^2}{u}g\right\rangle
            -\frac{2}{n}\int_M u^{-1}|\nabla u|^2\Delta u.
        \end{align*}
        Rearrange it and the proof finishes.
    \end{proof}

\begin{proof}[Proof of \cref{prop.weak perturbed eigenvalue estimate}]
        First assume that $s\in(-\infty,-2)\cup(-2,0]\cup [\frac{4(n+2)}{4n-1},+\infty)$. Let $t:=s+2\in(-\infty,0)\cup(0,2]\cup[\frac{6(2n+1)}{4n-1},+\infty)$ and $u:=v^{\frac{t}{2}}$. Then we have
    \begin{align*}
        v^{\frac{s}{2}}&=u^{-\frac{2}{t}+1},\\
        v^{\frac{s}{2}}(\nabla^2 v-\frac{\Delta v}{n}g)
    &=\frac{2}{t}
    \left(\nabla^2 u-\frac{\Delta u}{n}g\right)
    +\frac{2}{t}\left(\frac{2}{t}-1\right)
    \left(\frac{du\otimes du}{u}-\frac{1}{n}\frac{|\nabla u|^2}{u}g\right),\\
    v^{\frac{s}{2}}\Delta v&=\frac{2}{t}
    \left(\frac{2}{t}-1\right)
    u^{-1}|\nabla u|^2+\frac{2}{t}\Delta u,\\
    v^{\frac{s}{2}}\nabla v&=\frac{2}{t}\nabla u.
    \end{align*}
    By \cref{lem.laplace squared} , \cref{lem.grad squared laplace} and the fact that $|\frac{du\otimes du}{u}-\frac{1}{n}\frac{|\nabla u|^2}{u}g|^2=\frac{n-1}{n}u^{-2}|\nabla u|^4$, we have
    \begin{align*}
        &\int_M \left|\nabla^2 v-\frac{\Delta v}{n}g\right|^2 v^{s}
        +\frac{1}{n}\int_M (\Delta v)^2v^{s}\notag\\\notag
        =&\left(\frac{2}{t}\right)^2
        \int_M |\nabla^2 u-\frac{\Delta u}{n}g|^2
        +\left(\frac{2}{t}\right)^2
        \left(\frac{2}{t}-1\right)^2
        \frac{n-1}{n}
        \int_M u^{-2}|\nabla u|^4\\\notag
        &+2\left(\frac{2}{t}\right)^2
        \left(\frac{2}{t}-1\right)
        \int_M\left\langle\nabla^2 u-\frac{\Delta u}{n}g,\frac{du\otimes du}{u}-\frac{1}{n}\frac{|\nabla u|^2}{u}g\right\rangle\\\notag
        &+\frac{1}{n}
        \left(\frac{2}{t}\right)^2
        \left(\frac{2}{t}-1\right)^2
        \int_M u^{-2}|\nabla u|^4
        +\frac{1}{n}
        \left(\frac{2}{t}\right)^2\frac{n}{n-1}\int_M \left(\left|\nabla^2 u-\frac{\Delta u}{n}g\right|^2
        +\mathrm{Ric}(\nabla u,\nabla u)\right)\\\notag
        &+\frac{2}{n}\left(\frac{2}{t}\right)^2
        \left(\frac{2}{t}-1\right)
        \frac{n}{n+2}\int_M u^{-2}|\nabla u|^4\notag\\
        &-\frac{2}{n}
        \left(\frac{2}{t}\right)^2
        \left(\frac{2}{t}-1\right)
        \frac{2n}{n+2}\int_M \left\langle\nabla^2 u-\frac{\Delta u}{n}g,\frac{du\otimes du}{u}-\frac{1}{n}\frac{|\nabla u|^2}{u}g\right\rangle\\\notag
        =&\frac{n}{n-1}
        \left(\frac{2}{t}\right)^2
        \int_M \left|\nabla^2 u-\frac{\Delta u}{n}g\right|^2
        +\frac{2n}{n+2}
        \left(\frac{2}{t}\right)^2
        \left(\frac{2}{t}-1\right)
        \int_M \left\langle\nabla^2 u-\frac{\Delta u}{n}g,\frac{du\otimes du}{u}-\frac{1}{n}\frac{|\nabla u|^2}{u}g\right\rangle\\\notag
        &+\left(\left(\frac{2}{t}\right)^2\left(\frac{2}{t}-1\right)^2
        +\frac{2}{n+2}\left(\frac{2}{t}\right)^2
        \left(\frac{2}{t}-1\right)
        \right)
        \int_M u^{-2}|\nabla u|^4
        +\frac{1}{n-1}\left(\frac{2}{t}\right)^2
        \int_M \mathrm{Ric}(\nabla u,\nabla u)\\\notag
        =&\frac{n}{n-1}
        \left(\frac{2}{t}\right)^2\int_M\left|\nabla ^2u-\frac{\Delta u}{n}g
        +\left(\frac{2}{t}-1\right)\frac{n-1}{n+2}
        \left(\frac{du\otimes du}{u}-\frac{1}{n}\frac{|\nabla u|^2}{u}g\right)
        \right|^2\\\notag
        &+\frac{3(2n+1)}{(n+2)^2} 
        \left(\frac{2}{t}\right)^2 \left(\frac{2}{t}-1\right)\left(\frac{2}{t}-\frac{4n-1}{3(2n+1)}\right)\int_M u^{-2}|\nabla u|^4
        +\frac{1}{n-1}
        \left(\frac{2}{t}\right)^2\int_M \mathrm{Ric}(\nabla u,\nabla u)\\\notag
        =&\frac{n}{n-1}
        \left(\frac{2}{t}\right)^2
        \int_M\left|\nabla ^2u-\frac{\Delta u}{n}g
        +\left(\frac{2}{t}-1\right)
        \frac{n-1}{n+2}
        \left(\frac{du\otimes du}{u}-\frac{1}{n}\frac{|\nabla u|^2}{u}g\right)
        \right|^2\\ \notag
        &+\frac{4n-1}{4(n+2)^2}s
        \left(s-\frac{4(n+2)}{4n-1}\right)
        \int_M v^{s-2}|\nabla v|^4
        +\frac{1}{n-1}\int_Mv^s\mathrm{Ric}(\nabla v,\nabla v)\notag\\
        \geq&\frac{4n-1}{4(n+2)^2}s
        \left(s-\frac{4(n+2)}{4n-1}\right)
        \int_M v^{s-2}|\nabla v|^4
        + k\int_M |\nabla v|^2v^{s}
        \geq k\int_M |\nabla v|^2v^{s}.
    \end{align*}
    Hence we have proved \cref{weak perturbed eigenvalue estimate} for $s\in (-\infty, -2)\cup(-2,0]\cup [\frac{4(n+2)}{4n-1},+\infty)$.
    
    The case $s=-2$ could be obtained by taking the limit $s\to -2$ in \cref{weak perturbed eigenvalue estimate}.
    \end{proof}    
\end{document}